\documentclass[runningheads]{llncs}
\usepackage[utf8]{inputenc}
\usepackage{graphicx}
\usepackage{hyperref}
\usepackage{amssymb}
\usepackage{amsmath}
\usepackage{array}
\usepackage{color}
\usepackage{tkz-graph}
\usepackage{subfigure}

\graphicspath{{images/}}

\begin{document}

\title{Max-plus Operators Applied to Filter Selection and Model
  Pruning in Neural Networks}

\titlerunning{Max-plus Operators Applied to Filter Selection and Model
  Pruning}

\author{Yunxiang Zhang\inst{1,2,3} \and Samy Blusseau\inst{3} \and
  Santiago Velasco-Forero\inst{3} \and Isabelle Bloch\inst{2} \and
  Jesús Angulo\inst{3}}

\authorrunning{Y. Zhang, S. Blusseau, S. Velasco-Forero, I. Bloch, J. Angulo}

\institute{Ecole Polytechnique, 91128 Palaiseau, France \and LTCI,
  Télécom ParisTech, Université Paris-Saclay, Paris, France \and
  Centre for Mathematical Morphology, Mines ParisTech, PSL Research
  University, France}

\maketitle              


\begin{abstract}
  Following recent advances in morphological neural networks, we
  propose to study in more depth how Max-plus operators can be
  exploited to define morphological units and how they behave when
  incorporated in layers of conventional neural networks. Besides
  showing that they can be easily implemented with modern machine
  learning frameworks, we confirm and extend the observation that a
  Max-plus layer can be used to select important filters and reduce
  redundancy in its previous layer, without incurring performance
  loss. Experimental results demonstrate that the filter selection
  strategy enabled by a Max-plus layer is highly efficient and robust,
  through which we successfully performed model pruning on two neural
  network architectures. We also point out that there is a close
  connection between Maxout networks and our pruned Max-plus networks
  by comparing their respective characteristics. The code for
  reproducing our experiments is available online\footnote{For code
    release, please visit https://github.com/yunxiangzhang.}.

\keywords{Mathematical morphology \and Morphological neural networks~\and Max-plus operator \and Deep learning \and Filter Selection \and Model Pruning.}
\end{abstract}


\section{Introduction}

During the previous era of high interest in artificial neural
networks, in the late 1980s, morphological
networks~\cite{davidson1990theory,wilson1989morphological} were
introduced merely as an alternative to their linear counterparts. In a
nutshell, they consist in changing the elementary neural operation
from the usual scalar product to dilations, erosions and more general
rank filters. Besides the practical achievements in this research
line, which reached state-of-the-art results at its
time~\cite{pessoa2000neural}, it questioned the main practices in the
field of artificial neural networks on crucial topics such as
architectures and optimization methods, improving their understanding.

After the huge technological leap taken by deep neural networks in the
last decade, understanding them remains challenging and insights can
still be brought by testing alternatives to the most popular
practices. However, since the emergence of highly efficient deep
learning frameworks (Caffe, Torch, TensorFlow, \textit{etc.}),
morphological neural networks have been very little
re-investigated~\cite{Maxplus,mondal2019dense}. Motivated by the
promising results in~\cite{Maxplus}, this paper proposes to extend and
deepen the study on morphological neural networks with Max-plus
layers. More precisely, we aim to validate and exploit a specific
property, namely that Max-plus layer (\textit{i.e.} a dilation layer)
following a conventional linear layer tends to select a reduced number
of filters from the latter, making the others useless.

Our findings and contributions are three-fold. First, we propose an
efficient training framework for morphological neural networks with
Max-plus blocks, and theoretically demonstrate that they are universal
function approximators under mild conditions. Secondly, we perform
extensive experiments and visualization analysis to validate the
robustness and effectiveness of the filter selection property provided
by Max-plus blocks. Thirdly, we successfully applied the resulting
Max-plus blocks to the task of model pruning on different neural
network architectures. Related work is briefly summarized in
Section~\ref{sec:biblio}. Then we show how to introduce Max-plus
operators in neural networks in
Section~\ref{sec:maxplus}. Experimental results are discussed in
Section~\ref{sec:results}.


\section{Related Work}
\label{sec:biblio}

\textbf{Morphological neural networks} were defined almost
simultaneously in two different ways at the end of the
1980s~\cite{davidson1990theory,wilson1989morphological}. Davidson~\cite{davidson1990theory}
introduced neural units that can be seen as pure dilations or
erosions, whereas Wilson~\cite{wilson1989morphological} focused on a
more general formulation based on rank filters, in which $\max$ and
$\min$ operators are two particular cases.  Davidson's definition
interprets morphological neurons as bounding boxes in the feature
space~\cite{Ritter96,Ritter03,Sussner11}. In the latter studies,
networks were trained to perform perfectly on training sets after few
iterations, but little attention was drawn to generalization. Only
recently, a backpropagation-based algorithm was adopted and improved
constructive ones~\cite{zamora2017dendrite}. Still, the
``bounding-box'' approach does not seem to generalize well to test set
when faced with high-dimensional problems like image analysis.

Wilson's idea, on the other hand, inspired hybrid linear/rank filter
architectures, which were trained by gradient descent~\cite{Pessoa98}
and backpropagation~\cite{pessoa2000neural}. In this case, the
geometrical interpretation of decision surfaces becomes much richer,
and the resulting framework was successfully applied to an image
classification problem. The previously mentioned study~\cite{Maxplus}
is one of the latest in this area, introducing a hybrid architecture
that experimentally shows an interesting property on network pruning.
In this classification experiment~\cite{Maxplus}, each Max-plus unit
shows, after training, one or two large weight values compared to the
others. At the inference stage, this non-uniform distribution of
weight values induces a selection of important filters in the previous
layer, whereas the other filters (the majority) are no longer used in
the subsequent classification task. Therefore, after removal of the
redundant filters, the network behaves exactly as it did before
pruning. In this paper, we focus on the exploration of this property
and show that it becomes more stable and effective provided that a
proper regularization is applied.\\
\textbf{Model pruning} stands for a family of algorithms that explore
the redundancy in model parameters and remove unimportant ones while
not compromising the performance. Rapidly increased computational and
storage requirements for deep neural networks in recent years have
largely motivated research interests in this domain. Early works on
model pruning dates back to~\cite{hassibi1993,lecun1990}, which prune
model parameters based on the Hessian of the loss function. More
recently, the model pruning problem was addressed by first dropping
the neuronal connections with insignificant weight value and then
fine-tuning the resulting sparse network~\cite{han2015a}. This
technique was later integrated into the Deep Compression
framework~\cite{han2015b} to achieve even higher compression rate. The
HashNets model~\cite{chen2015} randomly groups parameters into hash
buckets using a low-cost hash function and performs model compression
\textit{via} parameter sharing. While previous methods achieved
impressive performance in terms of compression rate, one notable
disadvantage of these non-structured pruning algorithms lies in the
sparsity of the resulting weight matrices, which cannot lead to
speedup without dedicated hardwares or libraries.

Various structured pruning methods were proposed to overcome this
subtlety in practical applications by pruning at the level of channels
or even layers. Filters with smaller $L_{1}$ norm were pruned based on
a predefined pruning ratio for each layer in~\cite{li2016}. Model
pruning was transformed into an optimization problem in~\cite{luo2017}
and the channels to remove were determined by minimizing next layer
reconstruction error. A branch of algorithms in this category employs
$L_{p}$ regularization to induce shrinkage and sparsity in model
parameters. Sparsity constraints were imposed in~\cite{liu2017} on
channel-wise scaling factors and pruning was based on their magnitude,
while in~\cite{wen2016} group-sparsity was leverage to learn compact
CNNs via a combination of $L_{1}$ and $L_{2}$ regularization. One
minor drawback of these regularization-based methods is that the
training phase generally requires more iterations to converge. Our
approach also falls into the structured pruning category and thus no
dedicated hardware or libraries are required to achieve speedup, yet
no regularization is imposed during model training. Moreover, in
contrast to most existing pruning algorithms, our method does not need
fine-tuning to regain performance.


\section{Max-plus Operator as a Morphological Unit}
\label{sec:maxplus}


\subsection{Morphological Perceptron}
\label{subsection31}

In traditional literature on machine learning and neural networks, a
perceptron~\cite{MLP} is defined as a linear computational unit,
possibly followed by a non-linear activation function. Among all
popular choices of activation functions, such as logistic function,
hyperbolic tangent function and rectified linear unit (ReLU) function,
ReLU~\cite{ReLU} generally achieves better performance due to its
simple formulation and non-saturating property. Instead of
multiplication and addition, the morphological perceptron employs
addition and maximum, which results in a non-linear computational
unit. A simplified version~\cite{Maxplus} of the initial
formulation~\cite{davidson1990theory,Ritter96} is defined as follows.

\begin{definition}
  \textbf{\textup{(Morphological Perceptron).}}  Given an input vector
  $\mathbf{x} \in \mathbb{R}_{max}^{n}$ (with
  $\mathbb{R}_{max} = \mathbb{R} \cup \left\{-\infty\right\}$), a
  weight vector $\mathbf{w} \in \mathbb{R}_{max}^{n}$, and a bias
  $\mathbf{b} \in \mathbb{R}_{max}$, the morphological perceptron
  computes its activation as:
\begin{equation} 
  a(\mathbf{x}) = \mathop{\max}\left\{\mathbf{b}, \mathop{\max}_{i \in \left\{1, ... , n\right\}} \left\{\mathbf{x}_{i} + \mathbf{w}_{i}\right\}\right\}
  \label{eq:morph_perceptron}
\end{equation}
where $\mathbf{x}_i$ (resp. $\mathbf{w}_i$) denotes the $i$-th
component of $\mathbf{x}$ (resp. $\mathbf{w}$).
\end{definition}

This model may also be referred to as $(\max, +)$ perceptron since it
relies on the $(\max, +)$ semi-ring with underlying set
$\mathbb{R}_{max}$. It is a dilation on the complete lattice
$\big((\mathbb{R}\cup \lbrace\pm\infty\rbrace)^n, \leq_n\big)$ with
$\leq_n$ the Pareto ordering.

\subsection{Max-plus Block}
\label{subsection32}

Based on the formulation of the morphological perceptron, we define
the Max-plus block as a standalone module that combines a
fully-connected layer (or convolutional layer) with a Max-plus
layer~\cite{Maxplus}. Let us denote the input vector of the
fully-connected layer\footnote{This formulation can be easily
  generalized to the case of convolutional layers.}, the input and
output vectors of the Max-plus layer respectively by $\mathbf{x}$,
$\mathbf{y}$ and $\mathbf{z}$, whose components are indexed by
$i \in \left\{1, ... , I\right\}$, $j \in \left\{1, ... , J\right\}$
and $k \in \left\{1, ... , K\right\}$, respectively. The corresponding
weight matrices are denoted by
$\mathbf{w}^{f} \in \mathbb{R}_{max}^{I \times J}$ and
$\mathbf{w}^{m} \in \mathbb{R}_{max}^{J \times K}$. Then the operation
performed in this Max-plus block is (see Figure~\ref{fig:proof}):
\begin{equation} 
\begin{array}{rcl}
  \mathbf{y}_{j} & = & \sum_{i \in \left\{1, ... , I\right\}} \mathbf{x}_{i} \cdot \mathbf{w}_{ij}^{f}  \\
  \mathbf{z}_{k} & = & \mathop{\max}_{j \in \left\{1, ... , J\right\}} \left\{\mathbf{y}_{j} + \mathbf{w}_{jk}^{m}\right\}
\end{array}
\label{eq:maxplus1}
\end{equation}
Note that the bias vector of the fully-connected layer (convolutional
layer) is removed in our formulation, since its effect overlaps with
that of the weight matrix $\mathbf{w}^{m}$. In addition, the bias
vector of the Max-plus layer is shown to be ineffective in practice
and is therefore not used here.


\subsection{Universal Function Approximator Property}
\label{subsection33}

The result presented here is very similar to the approximation theorem
on Maxout networks\footnote{Note that the classical universal
  approximation theorems for neural networks (see for
  example~\cite{hornik1989multilayer}) do not hold for networks
  containing max-plus units.}~\cite{Maxout}, based on Wang's
work~\cite{wang2004general}. As shown in~\cite{Maxout}, Maxout
networks with enough affine components in each Maxout unit are
universal function approximators. Recall that a model is called a
universal function approximator if it can approximate arbitrarily well
any continuous function provided enough capacity. Similarly, provided
that the input vector (or input feature maps)
$\mathbf{y} \in \mathbb{R}_{max}^{J}$ of the Max-plus layer may have
arbitrarily many affine components (or affine feature maps), we show
that a Max-plus model with just two output units in its Max-plus block
can approximate arbitrarily well any continuous function of the input
vector (or input feature maps) $\mathbf{x} \in \mathbb{R}^{I}$ of the
block on a compact domain. A diagram illustrating the basic idea of
the proof is shown in Figure \ref{fig:proof}.

\begin{figure}[htbp]
\centering
\includegraphics[width=0.5\textwidth]{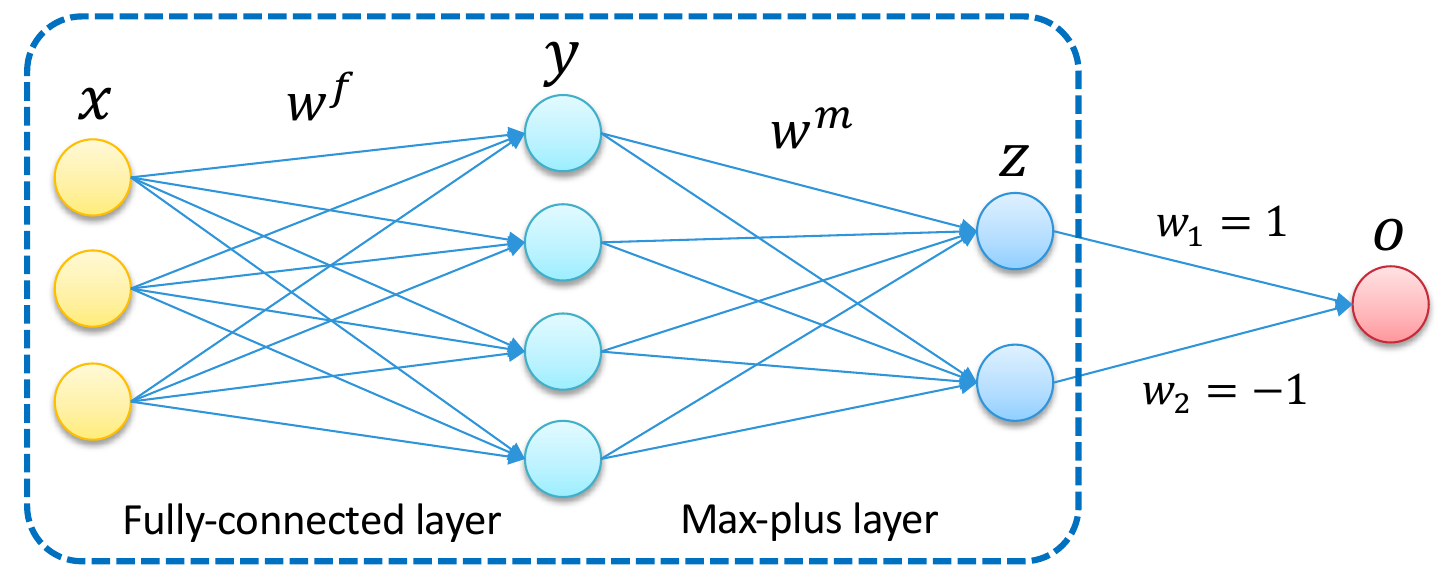}
\caption{A Max-plus block with two output units, applied to an input vector $x$.} 
\label{fig:proof}
\end{figure}

\begin{theorem}
  \textbf{\textup{(Universal function approximator).}}  A Max-plus
  model with two output units in its Max-plus block can approximate
  arbitrarily well any continuous function of the input of the block
  on a compact domain.
\end{theorem}

\begin{proof}
  We provide here a sketch of the proof, which follows very closely
  the one of Theorem 4.3 in~\cite{Maxout}. By Proposition 4.2
  in~\cite{Maxout}, any continuous function defined on a compact
  domain $C \subset \mathbb{R}^{I}$ can be approximated arbitrarily
  well by a piecewise linear (PWL) continuous function $g$, composed
  of $k$ affine regions. By Proposition 4.1 in~\cite{Maxout}, there
  exist two matrices $W_1, W_2 \in \mathbb{R}^{k\times I}$ and two
  vectors $b_1, b_2\in\mathbb{R}^k$ such that
\begin{equation}
  \label{eq:affine-approx}
  \forall x\in C, \; g(x) = \max_{1\leq j\leq k}\lbrace W_{1j}x + b_{1j}\rbrace - \max_{1\leq j\leq k}\lbrace W_{2j}x + b_{2j}\rbrace,
\end{equation}
where $W_{ij}$ is the $j$-th row of matrix $W_i$ and $b_{ij}$ the
$j$-th coefficient of $b_i$, $i=1,2$. Now,
Equation~\ref{eq:affine-approx} is the output of the Max-plus block of
Figure~\ref{fig:proof}, provided
$\mathbf{w}^{f} = [W_1; W_2] \in \mathbb{R}^{2k\times I}$ is the
matrix of the fully connected layer,
$\mathbf{w}_1^{m} = [b_1^T, -\infty, \dots, -\infty]\in
\mathbb{R}_{max}^{2k}$ and
$\mathbf{w}_2^{m} = [-\infty, \dots, -\infty, b_2^T]\in
\mathbb{R}_{max}^{2k}$ the two rows of $\mathbf{w}^{m}$. This
concludes the proof.
\end{proof}


\section{Experiments}
\label{sec:results}

In this section, we present experimental results for our Max-plus
blocks by integrating them in different types of neural networks. All
neural network models shown in this section are implemented with the
open-source machine learning library TensorFlow~\cite{TF} and trained
on benchmark datasets MNIST~\cite{CNN} or CIFAR-10~\cite{CIFAR}
depending on the model complexity and capability.


\subsection{Filter Selection Property}
\label{subsection41}

In an attempt to reproduce and confirm the experimental results
reported in~\cite{Maxplus}, we first implemented a simple Max-plus
model composed of a fully-connected layer with $J$ units followed by a
Max-plus layer with ten units, namely a Max-plus block in our
terminology, to perform image classification on MNIST dataset. In
contrast to the original formulation in~\cite{Maxplus}, the two bias
vectors are removed for practical concerns explained in
Section~\ref{subsection32}.

Table~\ref{tab:mnist} summarizes the classification accuracy achieved
by this simple model on the validation set of MNIST dataset for
different values of $J$ in columns 3 to 5. Note that all the
experiments contained in these 3 columns are conducted under the same
training setting (initial learning rate, learning rate decay steps,
batch size, optimizer, \textit{etc.}) except for parameter
initialization (each column corresponds to a different random
seed). The performance of the original model in~\cite{Maxplus} is
included in column 2 for comparison.  As shown in
Table~\ref{tab:mnist}, provided a proper initialization, our simple
Max-plus model generally achieves an improved performance compared to
the original model. More interestingly, through horizontal comparison
across different runs, we find that the performance of this naive
Max-plus model is highly sensitive to parameter initialization.

\begin{table}[htbp]
\centering
\caption{Classification accuracy on the validation set of MNIST dataset with three different random seeds.}
\label{tab:mnist}
{\small\begin{tabular}{|c|c|c|c|c|c|c|c|}
\hline
\, $J$ units \, & Model \cite{Maxplus} & \multicolumn{3}{c|}{Max-plus} & \multicolumn{3}{c|}{Max-plus + dropout} \\
\hline
24 & \hspace{0.2cm} 84.3\% \hspace{0.2cm} & \hspace{0.2cm} 76.7\% \hspace{0.2cm} & \hspace{0.2cm} \textbf{84.4\%} \hspace{0.2cm} & \hspace{0.2cm} 76.0\% \hspace{0.2cm} & \hspace{0.2cm} 93.7\% \hspace{0.2cm} & \hspace{0.2cm} 93.7\% \hspace{0.2cm} & \hspace{0.2cm} \textbf{94.2\%} \hspace{0.2cm}\\
\hline
32 & 84.8\% & 84.5\% & 76.5\% & \textbf{94.0\%} & \textbf{94.6\%} & 94.2\% & 93.7\% \\
\hline
48 & 84.6\% & \textbf{94.0\%} & 93.8\% & 75.8\% & \textbf{94.6\%} & 94.5\% & 94.5\% \\
\hline
64 & 92.1\% & 85.4\% & 94.7\% & \textbf{94.9\%} & 94.8\% & \textbf{95.1\%} & 94.8\% \\
\hline
100 & -- & \textbf{94.8\%} & 85.3\% & 93.7\% & 94.8\% & \textbf{95.5\%} & 95.2\% \\
\hline
144 & -- & \textbf{95.1\%} & 85.8\% & 85.4\% & 95.3\% & 95.7\% & \textbf{95.9\%} \\
\hline
\end{tabular}
}
\end{table}

In order to gain more insight into this instability problem, we follow
the approach of~\cite{Maxplus} and visualize the weight matrix of the
Max-plus layer $\mathbf{w}^{m} \in \mathbb{R}_{max}^{J \times 10}$ by
splitting and reshaping it into 10 gray-scale images
$\mathbf{w}_{\cdot 1}^{m} \, ... \, \mathbf{w}_{\cdot i}^{m} \, ... \,
\mathbf{w}_{\cdot 10}^{m}$, each corresponding to a specific class
(Figure~\ref{fig:weights_filters}, left). In addition we also
visualized, as on the right hand side of
Figure~\ref{fig:weights_filters}, the ten linear filters from the
fully-connected layer that correspond to the maximum value of each
weight vector $\mathbf{w}_{\cdot i}^{m}$, defined as:
\begin{equation} {\rm Filter}_{i} = \mathbf{w}_{\cdot
    j_{\max}^{(i)}}^{f} \in \mathbb{R}_{max}^{I} \quad {\rm where}
  \hspace{0.2cm} j_{\max}^{(i)} = \mathop{\arg\max}_{j \in \left\{1,
      ... , J\right\}}\left\{\mathbf{w}_{ji}^{m}\right\}.
\label{eq:filters}
\end{equation}

Figure~\ref{fig:weights_filters} shows specifically the weights of the
Max-plus model that achieves an accuracy of $85.8\%$ with $J = 144$
units (fifth column and last row in Table~\ref{tab:mnist}). We notice
that there exists a severe filter-collision problem in the Max-plus
block, namely different output units select the same linear filter in
the fully-connected layer to compute their outputs. More specifically,
class $3$ and class $8$ share the same $argmax$ (highlighted by red
circles) in their weight vectors and consequently select the same
linear filter (visualized filters for class $3$ and class $8$ are the
same). This collision between output units directly leads to
classification confusion (because the Max-plus layer is the last layer
in this simple model) since the classes that employ the same linear
filter are completely indistinguishable. Furthermore, we only observed
this filter-collision problem in the experiments that achieve
relatively poor performance compared to the others (same model with
different initialization). This observation consistently verifies our
hypothesis that filter-collision is at the root of lower performance.

\begin{figure}[htbp]
\centering
\begin{minipage}[t]{0.49\linewidth}
\centering
\includegraphics[width=0.95\linewidth]{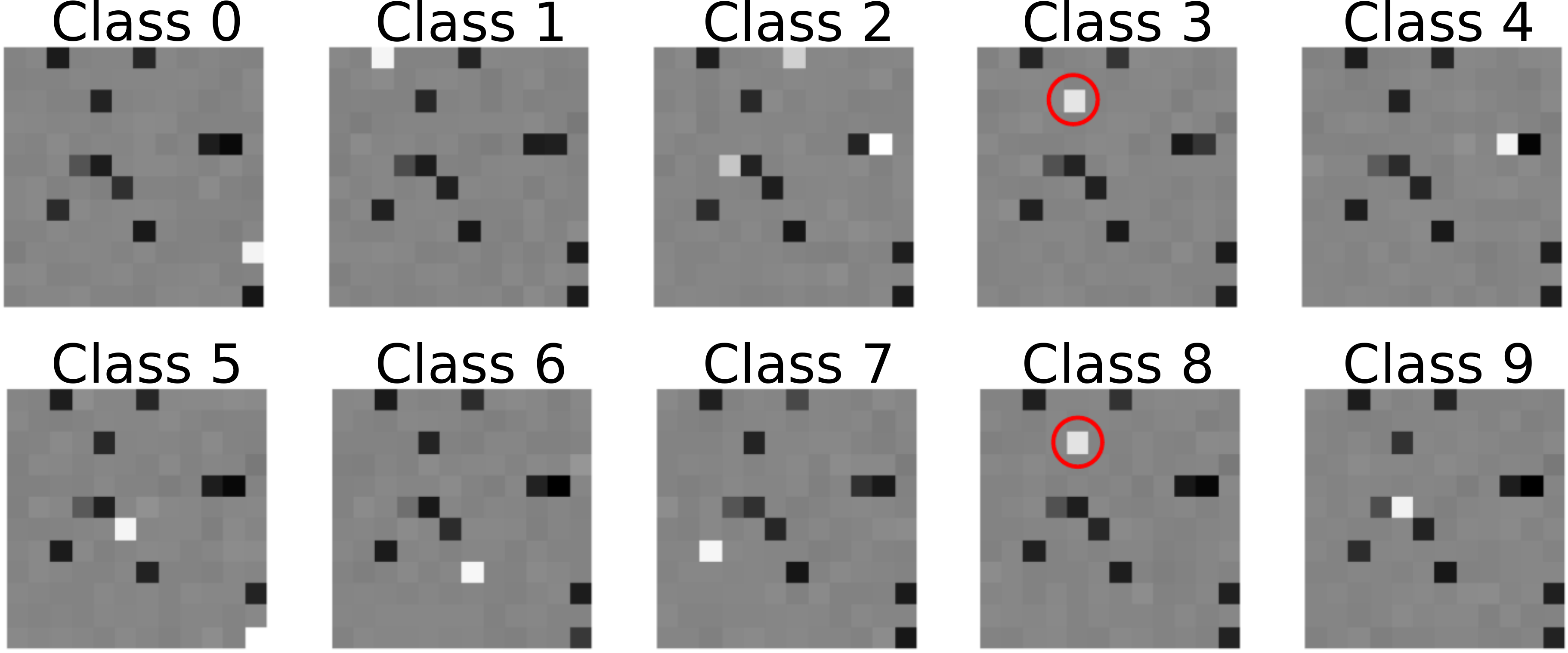}
\end{minipage}
\begin{minipage}[t]{0.49\linewidth}
\centering
\includegraphics[width=0.95\linewidth]{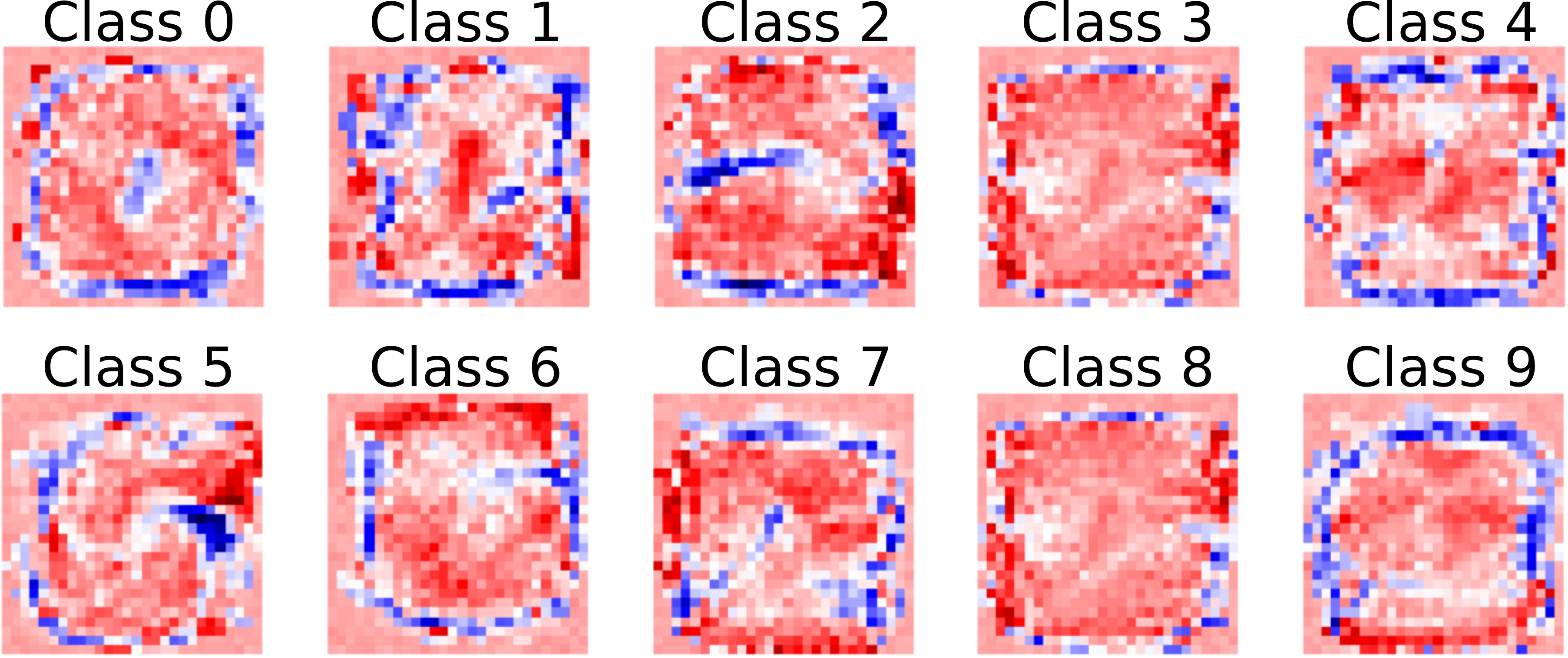}
\end{minipage}
\caption{Visualization of the weight matrix $\mathbf{w}^{m}$ (left)
  and the 10 linear filters that correspond to the maximum value of
  each weight vector $\mathbf{w}_{\cdot i}^{m}$ (right).}
\label{fig:weights_filters}
\end{figure}

In order to separate the output units that got stuck with the same
linear filter, we applied a dropout regularization
\cite{srivastava2014dropout} to randomly switch off the neuronal
connections between the fully-connected layer and the Max-plus layer
during training. Empirically, we found this approach highly effective,
as the performance of the dropout-regularized Max-plus model becomes
much less sensitive to parameter initialization. Table~\ref{tab:mnist}
shows that the Max-plus model with dropout regularization achieves
both a better performance and more stable results. We further validate
the effect provided by dropout regularization through an additional
experiment. With $J = 144$, we carried out 25 runs with different
initializations for different dropout ratios. As shown in
Figure~\ref{fig:MnistDropoutCurves}, dropout reduces dramatically the
variability across experiments. The interpretation of this improvement
is that dropout forces each class to use more than one linear filter
to represent its corresponding digit. This allows for a more general
representation and limits the risk of collisions between
classes. Interestingly, we observe a slight accuracy drop when dropout
ratio surpasses a certain level. This indicates that a trade-off
between stability and performance needs to be found, although a large
range of dropout values (between 25\% and 75\%) show robustness to
random seed values without penalizing the effectiveness on
classification accuracy.

\begin{figure}[htbp]
    \centering
    \includegraphics[width=1.\textwidth]{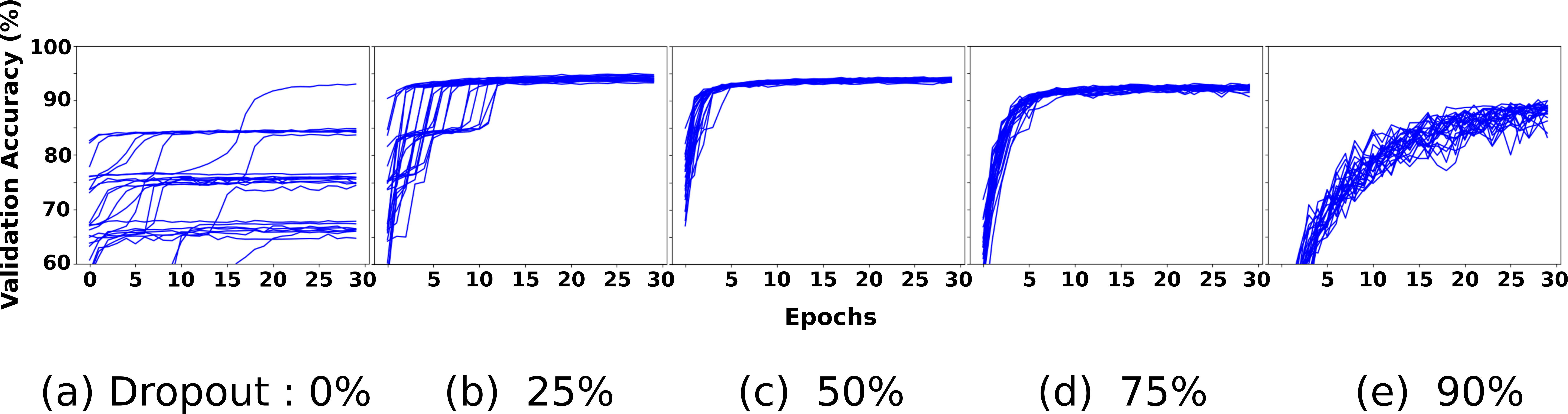}
    \caption{Classification accuracy per epochs for 25 runs with
      different random initialization and dropout ratios on MNIST
      validation set. Dropout values between 25\% and 75\% in (b-d)
      show robustness to random seed values without penalizing the
      effectiveness on classification accuracy.}
    \label{fig:MnistDropoutCurves}
\end{figure}
Whereas in general (with or without dropout) linear filters
$\mathbf{w}_{\cdot j^{(i)}}^{f}$ that correspond to large values of
$\mathbf{w}_{\cdot i}^{m}$ are similar to the images in
Figure~\ref{fig:weights_filters} (right), \textit{i.e.} digit-like
shape with high contrast, those corresponding to smaller values in the
weight matrix $\mathbf{w}^{m}$ are noisy or low-contrast images
showing the shape of a specific digit. This means that these filters
were activated by few training examples. Figure~\ref{fig:filters}
(left) shows several linear filters of this kind.
\begin{figure}[htbp]
\centering
\begin{minipage}[t]{0.49\linewidth}
\centering
\includegraphics[width=0.95\linewidth]{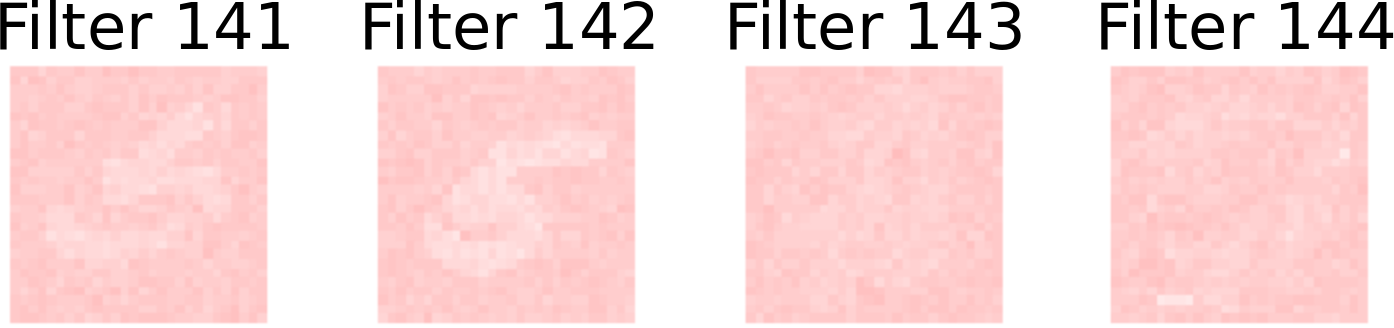}
\end{minipage}
\begin{minipage}[t]{0.49\linewidth}
\centering
\includegraphics[width=0.95\linewidth]{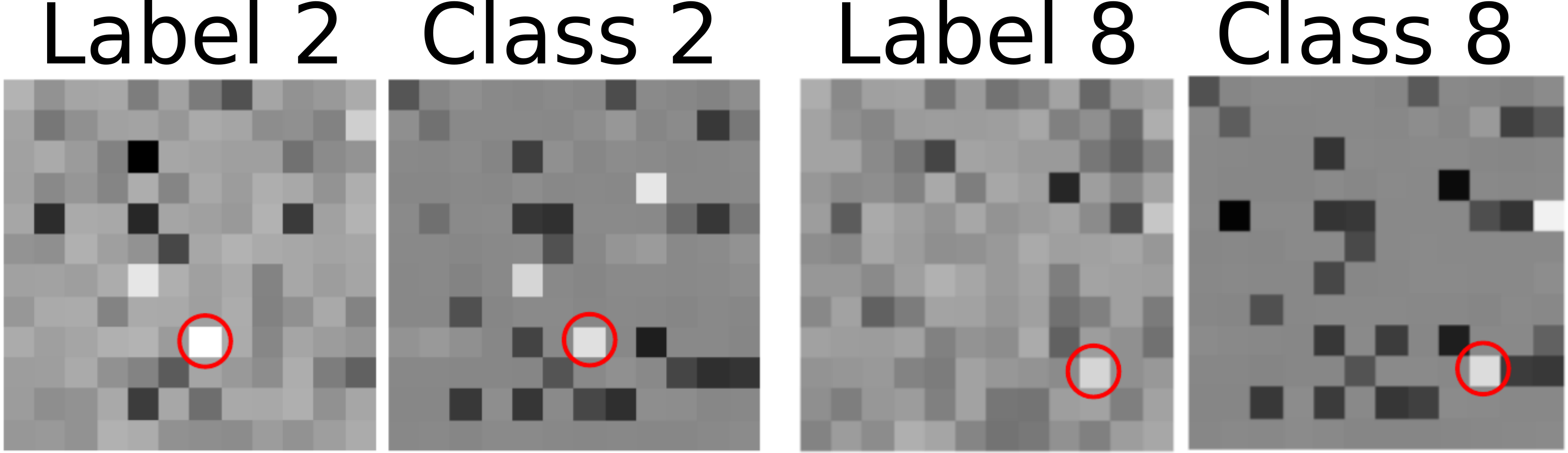}
\end{minipage}
\caption{Visualization of several linear filters corresponding to
  small values in $\mathbf{w}^{m}$ (left) and comparison between the
  activations of the fully-connected layer for two training examples
  and their corresponding weight vectors in $\mathbf{w}^{m}$ (right).}
\label{fig:filters}
\end{figure}
The visualization above suggests that large values in $\mathbf{w}^{m}$
correspond to linear filters of the fully-connected layer that contain
rich information for the subsequent classification task, and hence
that strongly respond to the samples of a specific class. If that
holds, then it is likely that only these filters shall achieve the
maximum in \eqref{eq:maxplus1}, and contribute to the classification
output. Therefore, we might be able to reduce the complexity of a
model while not degrading its performance by exploiting this filter
selection property of Max-plus layers.

In order to further validate the stability of this connection before
taking advantage of it, we also visualized the activations of the
fully-connected layer as gray-scale images for several training
examples and compare\footnote{For example, if a training example has
  label $i$, then we compare its activation vector
  $\mathbf{y} \in \mathbb{R}_{max}^{J}$ with the weight vector
  $\mathbf{w}_{\cdot (i+1)}^{m} \in \mathbb{R}_{max}^{J}$.} them to
the visualization of $\mathbf{w}^{m}$. As shown in Figure
~\ref{fig:filters} (right), there is a clear correspondence between
the maximum activation value of a training example with label $i$ and
the largest two or three values of the weight vector
$\mathbf{w}_{\cdot (i+1)}^{m}$, which indicates that the linear
filters corresponding to large values in the weight matrix
$\mathbf{w}^{m}$ are effectively used for the subsequent
classification task.


\subsection{Application to Model Pruning} 
\label{subsection42}

Now that our approach to filter selection via Max-plus layers is
proved to be quite effective and stable, we formalize our model
pruning strategy as follows: given a fixed threshold
$s \in \left[0, 1\right]$, for each weight vector
$\mathbf{w}_{\cdot (i+1)}^{m}$, we only keep the values that are
larger than
$s \times \mathop{\max}_{j \in \left\{1, ... , J\right\}}
\left\{\mathbf{w}_{j(i+1)}^{m}\right\} + \left(1 - s\right) \times
\mathop{\min}_{j \in \left\{1, ... , J\right\}}
\left\{\mathbf{w}_{j(i+1)}^{m}\right\}$ and the linear filters that
correspond to these retained values. Therefore, if in total $J_{r}$
linear filters are kept in the pruned model, then the remaining
parameters in the Max-plus layer no longer form a weight matrix but a
weight vector of size $J_{r}$, where each entry corresponds to a
linear filter in the fully-connected layer. The pruned fully-connected
layer and the pruned Max-plus layer combined together perform a
standard linear transformation followed by a maximum operation over
uneven groups. Note that the pruning process is conducted
independently for each output unit $\textbf{z}_{k}$ of the Max-plus
block ($1 \leq k \leq 10$), thus the number of retained linear filters
in the pruned model may vary from one class to another. From now on,
we shall call these selected linear filters by \emph{active filters}
and the others by \emph{non-active filters}. Figure~\ref{fig:pruned}
shows a graphical illustration of the comparison between the original
model and the pruned model.

\begin{figure}[htbp]
\centering
\includegraphics[width=0.8\linewidth]{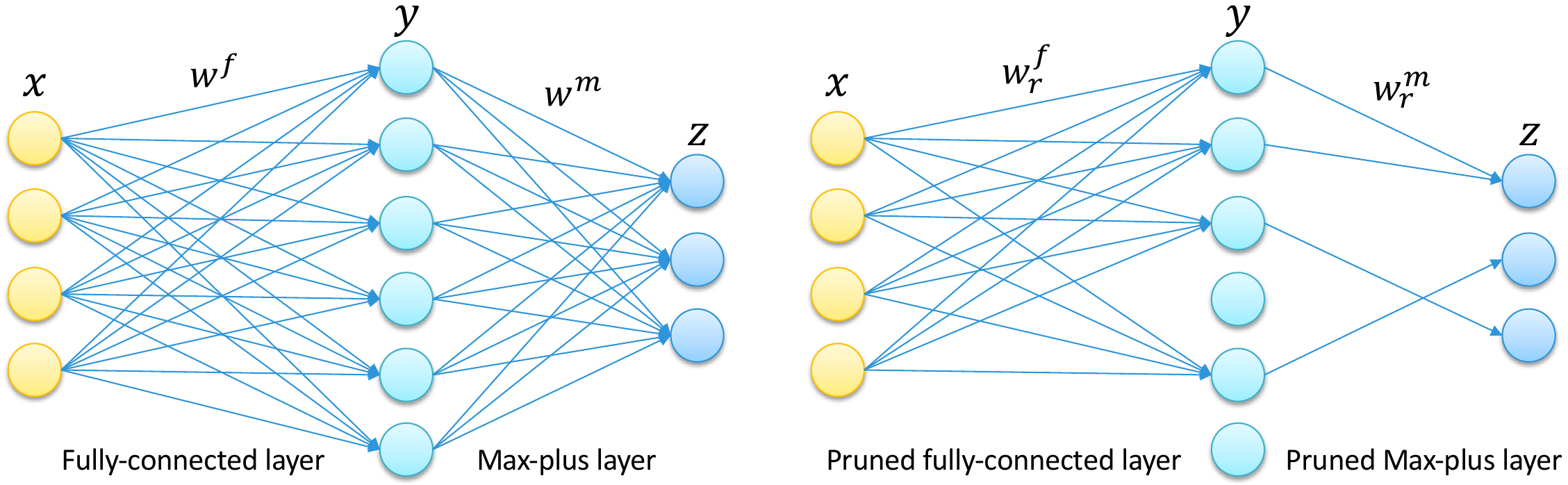}
\caption{Illustration of the comparison between the original Max-plus
  model (left) and the pruned Max-plus model (right). Here we have
  $\mathbf{w}_{r}^{f} \in \mathbb{R}_{max}^{I \times J_{r}}$ and
  $\mathbf{w}_{r}^{m} \in \mathbb{R}_{max}^{J_{r}}$.}
\label{fig:pruned}
\end{figure}

We tried different pruning levels on this simple Max-plus model by
varying the threshold $s$ and tested the pruned models on the
validation set and test set of MNIST dataset. We plotted the resulting
classification accuracy in function of the number of active filters in
Figure~\ref{fig:num_filters}. The performance of a single-layer
softmax model and a single-layer Maxout model (number of affine
components in each Maxout unit is two) is also provided for
comparison.

\begin{figure}[htbp]
\centering
\begin{minipage}[h]{0.39\linewidth}
\centering
{\small
\begin{tabular}{|c|c|c|c|}
\hline
\, $J_{r}$ \, & Accuracy & \, $J_{r}$ \, & Accuracy \\
\hline
10 & \hspace{0.2cm} 82.7\% \hspace{0.2cm} & 18 & \hspace{0.2cm} 90.4\% \hspace{0.2cm} \\
\hline
11 & \hspace{0.2cm} 86.4\% \hspace{0.2cm} & 19 & \hspace{0.2cm} 94.6\% \hspace{0.2cm} \\
\hline
12 & \hspace{0.2cm} 86.4\% \hspace{0.2cm} & 20 & \hspace{0.2cm} 94.7\% \hspace{0.2cm} \\
\hline
13 & \hspace{0.2cm} 89.9\% \hspace{0.2cm} & 21 & \hspace{0.2cm} 94.7\% \hspace{0.2cm} \\
\hline
14 & \hspace{0.2cm} 90.0\% \hspace{0.2cm} & 22 & \hspace{0.2cm} 94.8\% \hspace{0.2cm} \\
\hline
15 & \hspace{0.2cm} 90.2\% \hspace{0.2cm} & 23 & \hspace{0.2cm} 94.8\% \hspace{0.2cm} \\
\hline
16 & \hspace{0.2cm} 90.2\% \hspace{0.2cm} & 24 & \hspace{0.2cm} 95.7\% \hspace{0.2cm} \\
\hline
17 & \hspace{0.2cm} 90.3\% \hspace{0.2cm} & 25 & \hspace{0.2cm} 95.7\% \hspace{0.2cm} \\
\hline
\end{tabular}
}
\end{minipage}
\begin{minipage}[h]{0.5\linewidth}
\centering
\includegraphics[width=0.85\linewidth]{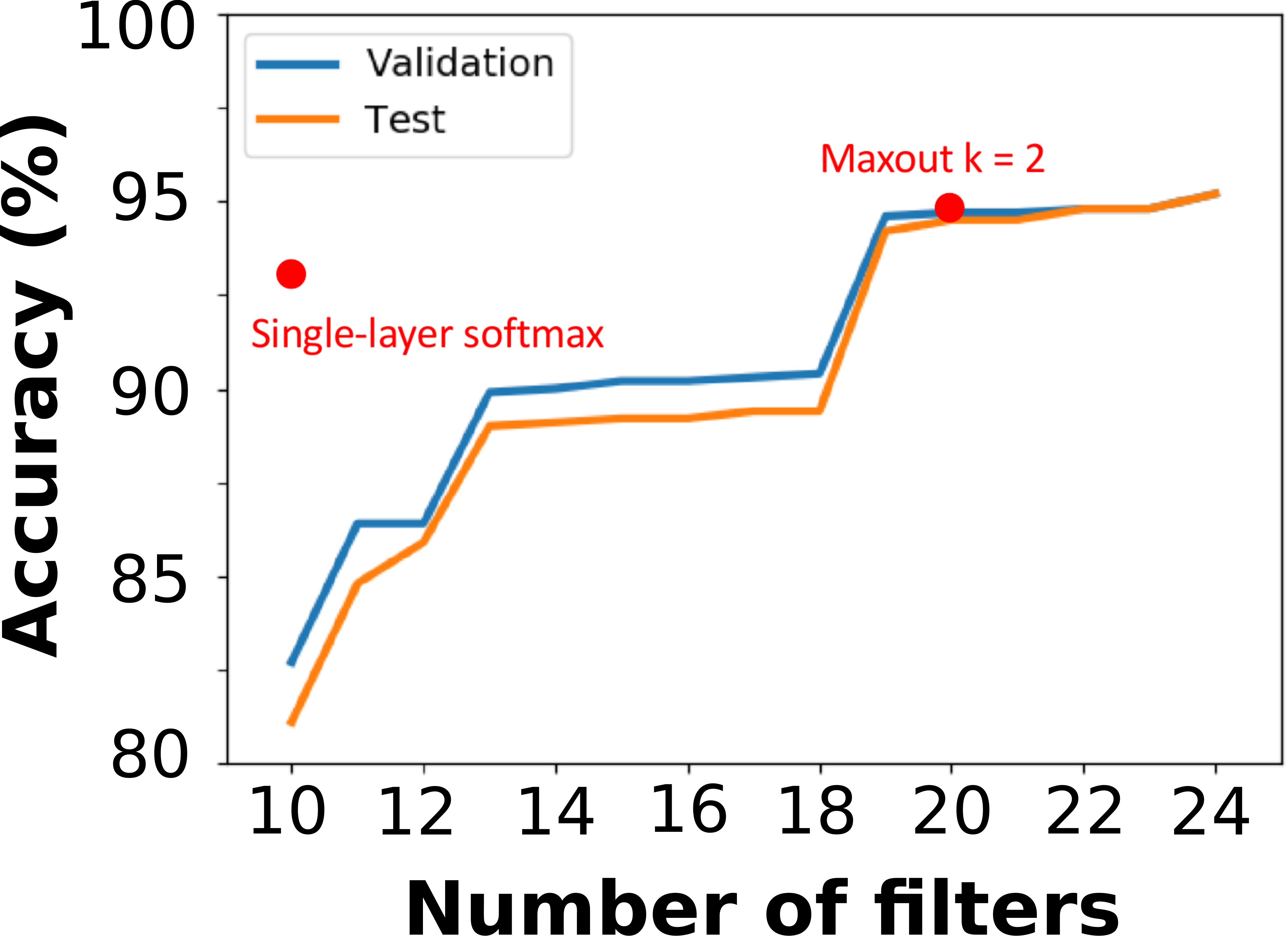}
\end{minipage}
\caption{Classification accuracy of pruned Max-plus model in function of the number of retained active filters.}
\label{fig:num_filters}
\end{figure}

As we can see on the diagram, the performance of the pruned Max-plus
model is quite inferior to that of a single-layer softmax model when
only one active filter is allowed to be selected for each class,
\textit{i.e.} the threshold is fixed to 1.0. However, as we relaxed
the constraint on the number of total active filters by decreasing the
threshold, the accuracy recovers rapidly and approaches that of the
unpruned Max-plus model in a monotonic way. With exactly 24 active
filters retained in the pruned model, we achieve a full-recovery of
the original Max-plus model performance, which means that the other
120 linear filters do not contribute to the classification
task. Moreover, we can achieve comparable performance as the 2-degree
Maxout model with roughly the same amount of parameters, which again
validates the effectiveness of our Max-plus models.

With the same method, we successfully performed model pruning on a
much more challenging CNN model by replacing the last fully-connected
layer with a Max-plus layer. In order to facilitate the training of
deep Max-plus model, we resort to transfer learning by initializing
the two convolutional layers with pre-trained weights. The pruned
Max-plus model achieves slightly better performance than the CNN model
while reducing $94.8\%$ of parameters of the second last
fully-connected layer and eliminating the last fully-connected layer
compared to the CNN model. Note that we could achieve a full-recovery
of the unpruned Max-plus model performance with only ten active
filters in this case, namely one linear filter for each output
unit. Table~\ref{tab:arch} summarizes the architectures of the CNN
model, the unpruned Max-plus model and the pruned Max-plus model,
along with their classification accuracy on the test set of CIFAR-10
dataset.

\begin{table}[htbp]
\centering
\caption{The architectural specifications of the CNN model, unpruned
  and pruned Max-plus model, along with their performance on the test
  set of CIFAR-10 dataset.}
\label{tab:arch}
\begin{tabular}{||c|c|c||}
\hline
\, CNN \, & \, Max-plus \, & \, Pruned Max-plus \, \\
\hline
conv(5*5) & conv(5*5) & conv(5*5) \\
\hline
\hspace{0.2cm} maxpool(2*2) \hspace{0.2cm} & \hspace{0.2cm} maxpool(2*2) \hspace{0.2cm} & \hspace{0.2cm} maxpool(2*2) \hspace{0.2cm} \\
\hline
conv(5*5) & conv(5*5) & conv(5*5) \\
\hline
maxpool(2*2) & maxpool(2*2) & maxpool(2*2) \\
\hline
fc(384) & fc(384) & fc(384) \\
\hline
fc(192) & fc(192) & fc(10) \\
\hline
fc(10) & maxplus(10) & maxplus(10) \\
\hline\hline
83.5\% & 83.9\% & 83.9\% \\
\hline
\end{tabular}
\end{table}


\subsection{Comparison to Maxout Networks}
\label{subsection43}

It is noticeable that the pruned Max-plus network differs from Maxout
networks only in their grouping strategy for maximum
operations. Maxout networks impose this rigid constraint in a way that
each Maxout unit has an equal number of affine components while
Max-plus networks are more tolerant in this
respect. Figure~\ref{fig:maxplus_maxout} shows a graphical
illustration of this comparison.

\begin{figure}[htbp]
\centering
\includegraphics[width=0.8\linewidth]{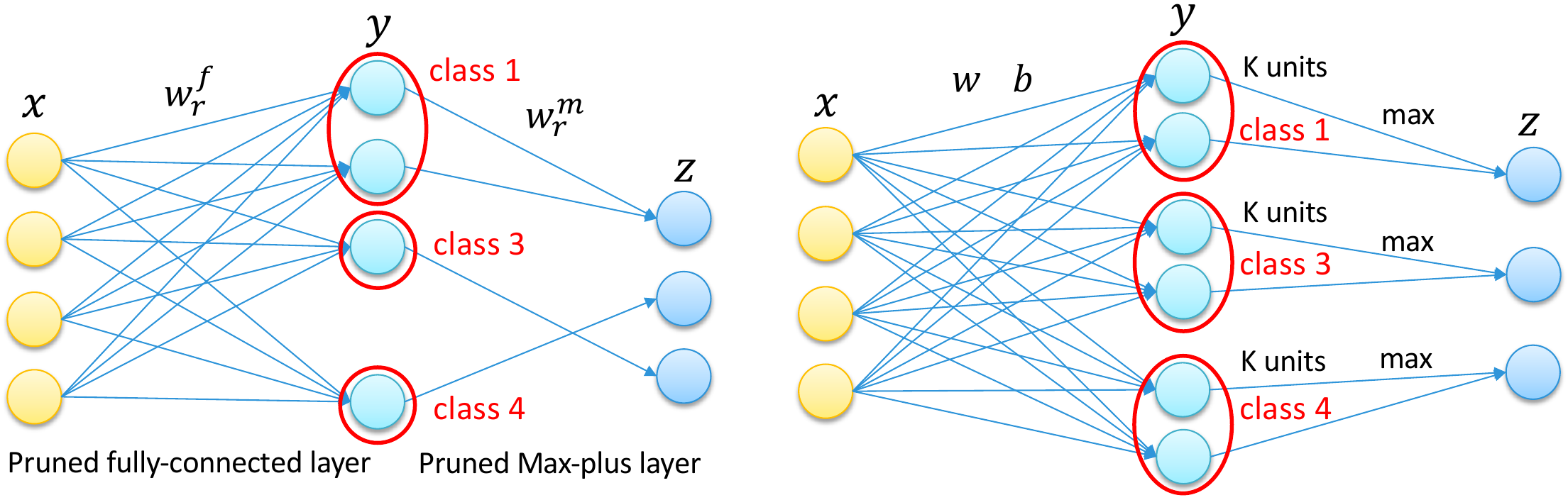}
\caption{Illustration of the comparison between the pruned Max-plus
  model (left) and Maxout model (right).}
\label{fig:maxplus_maxout}
\end{figure}

This higher flexibility hence endows Max-plus blocks with the
capability of adapting the number of active filters used for each
output unit accordingly. If a latent concept (say digit 1, which can
be easily confused with digit 7) is considerably tougher to capture
than some other concepts (say digits 3, 4 and 8, each of which has a
relatively unique shape among the ten Arabic numbers), then the
Max-plus layer will select more linear filters to abstract it than for
the others. For example, the partition of the 24 active filters for
the ten digit classes in Section~\ref{subsection42} is
$\left[2, 3, 3, 1, 2, 3, 3, 3, 2, 2\right]$, which is consistent with
our point. This adaptive behavior of the filter selection property of
Max-plus blocks makes the pruned Max-plus network more computationally
efficient (fewer model parameters, smaller run-time memory footprint
and faster inference) and is highly desirable in real-life
applications.


\section{Conclusions and Future Work}

In this work we went a step further on a very new and promising topic,
namely the reduction of deep neural networks with Max-plus blocks. Our
experiments show strong evidence that model pruning via this method is
compatible with high performance when a proper dropout regularization
is applied during training. This was tested on data and architectures
of variable complexity. Just as interesting as the obtained results
are the many questions raised by these new insights. In particular,
training these architectures is a challenging task which requires a
better understanding on them, both theoretical and practical. We
observed that training a deep model containing a Max-plus block is not
straightforward, as we needed to resort to transfer learning. New
optimization tricks will be needed to train deep architectures with
\emph{several} Max-plus blocks. The extension to a convolutional
version of Max-plus blocks is also an open question, which we hope can
be addressed based on the elements provided by this work.

\textbf{Acknowledgements.} 
This work was partially funded by a grant from Institut Mines Telecom.

\bibliographystyle{splncs04}
\bibliography{biblio}

\end{document}